\documentclass[a4paper]{amsart}
\usepackage{amssymb, enumitem}
\usepackage{hyperref, aliascnt}

 \usepackage{showlabels}

\usepackage{tikz}
\usetikzlibrary{calc}

\newcommand{\andSep}{\,\,\,\text{ and }\,\,\,}

\newcommand{\RR}{{\mathbb{R}}}
\newcommand{\ca}{$C^*$-algebra}

\newcommand{\calZ}{\mathcal{Z}}

\newcommand{\spec}{\sigma}
\newcommand{\charFct}{\mathbf{1}}
\newcommand{\axiomO}[1]{(O#1)}
\newcommand{\CatCu}{\ensuremath{\mathrm{Cu}}}
\newcommand{\CuSgp}{$\CatCu$-sem\-i\-group}

\DeclareMathOperator{\QT}{QT}
\DeclareMathOperator{\T}{T}
\DeclareMathOperator{\F}{F}
\DeclareMathOperator{\LL}{L}
\DeclareMathOperator{\Lsc}{Lsc}

\DeclareMathOperator{\Cu}{Cu}

\def\today{\number\day\space\ifcase\month\or   January\or February\or
   March\or April\or May\or June\or   July\or August\or September\or
   October\or November\or December\fi\   \number\year}

\newtheorem{lma}{Lemma}[section]

\newaliascnt{thmCt}{lma}
\newtheorem{thm}[thmCt]{Theorem}
\aliascntresetthe{thmCt}

\newaliascnt{corCt}{lma}
\newtheorem{cor}[corCt]{Corollary}
\aliascntresetthe{corCt}

\newaliascnt{prpCt}{lma}
\newtheorem{prp}[prpCt]{Proposition}
\aliascntresetthe{prpCt}

\theoremstyle{definition}

\newaliascnt{dfnCt}{lma}

\aliascntresetthe{dfnCt}

\newaliascnt{rmkCt}{lma}
\newtheorem{rmk}[rmkCt]{Remark}
\aliascntresetthe{rmkCt}

\newaliascnt{rmksCt}{lma}

\aliascntresetthe{rmksCt}

\newaliascnt{exaCt}{lma}

\aliascntresetthe{exaCt}

\newaliascnt{qstCt}{lma}

\aliascntresetthe{qstCt}

\newaliascnt{pgrCt}{lma}

\aliascntresetthe{pgrCt}

\newcounter{theoremintro}

\newaliascnt{thmIntroCt}{theoremintro}
\newtheorem{thmIntro}[thmIntroCt]{Theorem}
\aliascntresetthe{thmIntroCt}

\newaliascnt{dfnIntroCt}{theoremintro}

\aliascntresetthe{dfnIntroCt}

\newaliascnt{prpIntroCt}{theoremintro}

\aliascntresetthe{prpIntroCt}

\newaliascnt{corIntroCt}{theoremintro}
\newtheorem{corIntro}[corIntroCt]{Corollary}
\aliascntresetthe{corIntroCt}

\newaliascnt{qstIntroCt}{theoremintro}

\aliascntresetthe{qstIntroCt}

\newaliascnt{pbmIntroCt}{theoremintro}

\aliascntresetthe{pbmIntroCt}

\DeclareMathOperator{\Her}{Her}

\DeclareMathOperator{\Aff}{Aff}
\DeclareMathOperator{\Ped}{Ped}

\newcommand{\Cpct}{\mathcal{K}}

\hypersetup{hidelinks,pdftitle={Ranks of operators and pureness of C*-algebras},pdfauthor={Hannes Thiel}}
\title{Ranks of operators and pureness of $C^*$-algebras}
\author{Hannes Thiel}
\address{Hannes~Thiel,
Department of Mathematical Sciences, Chalmers University of Technology and University of
Gothenburg, Gothenburg SE-412 96, Sweden.}
\email{hannes.thiel@chalmers.se}
\urladdr{www.hannesthiel.org}
\thanks{The author was partially supported by the Knut and Alice Wallenberg Foundation (KAW 2021.0140) and by the Swedish Research Council project grant 2024-04200.}
\date{\today}

\begin{document}
\begin{abstract}
For simple, nonelementary \ca{s}, we show uniform density of functions realized as ranks of operators.
%We show that every strictly positive, continuous, affine function on the tracial state space of a simple, nonelementary \ca{} can be uniformly approximated by ranks of operators.
%This provides an approximate solution to the rank problem in this setting.
It follows that strict comparison by traces forces almost divisibility, and hence pureness.
As an application, we show that strict comparison implies $\mathcal{Z}$-stability for separable, simple, nonelementary \ca{s} with locally finite nuclear dimension.
In particular, the Toms--Winter conjecture holds for approximately subhomogeneous \ca{s}.

More generally, we obtain uniform density of ranks for \ca{s} with the Global Glimm Property.
Consequently, an exact \ca{} is pure if and only if it has strict comparison together with the Global Glimm Property.
\end{abstract}
\maketitle

%==========================================================================================
%==========================================================================================
\section{Introduction}
\label{sec:introduction}

%==========================================================================================
Let $A$ be a \ca{}.
The rank of a positive element $a\in A\otimes\mathcal{K}$ at a lower semicontinuous trace $\tau$ is
\[
d_\tau(a)=\lim_{n\to\infty}\tau(a^{1/n}).
\]
The resulting function $\tau\mapsto d_\tau(a)$ is lower semicontinuous, additive, and homogeneous on the trace cone.
The rank problem asks which functions can be obtained in this way.
It is usually formulated using $2$-quasitraces, which agree with traces for exact algebras.
We consider uniform approximation by ranks over ordinary traces.

Our first main result is the following.

%==========================================================================================
\begin{thmIntro}[\autoref{prp:RanksSimpleUnital}]
\label{intro:simple-ranks}
Let $A$ be a unital, simple, nonelementary \ca{}.
Assume that $T_1(A)\neq\emptyset$, and let $f\in\Aff(T_1(A))_{++}$ and $\varepsilon>0$.
Then there exist $m\geq1$ and $e\in M_m(A)_+$ such that
\[
\bigl| f(\tau)-d_\tau(e) \bigr| \leq \varepsilon
\]
for every $\tau \in T_1(A)$.
\end{thmIntro}

%==========================================================================================
Here $\Aff(K)_{++}$ denotes the strictly positive, continuous, affine functions on $K$, and matrix traces are unnormalized.
There are no assumptions of separability, exactness, stable rank one, or regularity of the extreme boundary.

Uniform rank approximation for simple unital nonelementary AH algebras is contained in the argument of Brown, Perera, and Toms in \cite[Section~5]{BroPerTom08CuElliottConj}; the relevant ingredients are Proposition~5.1 and Lemma~5.2 there.
Toms used this approximation to deduce almost divisibility from strict comparison and obtain $\mathcal{Z}$-stability for simple unital AH algebras; see \cite[Proposition~2.1 and Corollary~2.2]{Tom11CharClassifiableAH}.
\autoref{intro:simple-ranks} also answers the approximate rank question for approximately subhomogeneous algebras recorded in \cite[Problem~XXIV]{SchTikWhi25NuclCAlgs99Problems}.

Exact rank realization has been studied under additional regularity hypotheses.
In the presence of strict comparison, Dadarlat and Toms obtained exact realization results under additional hypotheses on the trace boundary or on the ranks already available in the algebra; see \cite[Theorem~1.1]{DadTom10Ranks}.
For separable, simple, unital, nonelementary algebras of stable rank one, exact realization over quasitraces was proved in \cite[Theorem~8.11]{Thi20RksOps}.
Further realization theorems for algebras of stable rank one, including nonsimple algebras, were obtained in \cite[Theorems~7.13 and~7.14]{AntPerRobThi22CuntzSR1}.
The present approximation theorem requires no stable rank hypothesis, but concerns ordinary traces.

\medskip

The rank approximation theorem supplies almost divisibility when combined with strict comparison.
Recall that a \ca{} is pure if its Cuntz semigroup is almost unperforated and almost divisible.
We distinguish strict comparison by ordinary traces from strict comparison by $2$-quasitraces; the latter is the usual unqualified comparison property.

%==========================================================================================
\begin{corIntro}[\autoref{prp:PureSimpleUnital}]
\label{intro:simple-pure}
Every unital, simple, nonelementary \ca{} with strict comparison by traces is pure.
\end{corIntro}

%==========================================================================================
For exact algebras, strict comparison by traces agrees with strict comparison.
Combining \autoref{intro:simple-pure} with Winter's $\calZ$-stability theorem \cite{Win12NuclDimZstable}, R{\o}rdam's pureness theorem \cite{Ror04StableRealRankZ}, and the nuclear dimension theorem of Castillejos--Evington--Tikuisis--White--Winter \cite{CasEviTikWhiWin21NucDimSimple}, gives the following application:

%==========================================================================================
\begin{thmIntro}[\autoref{prp:TomsWinter}, \autoref{rmk:ASH}]
\label{intro:toms-winter}
Let $A$ be a separable, unital, simple, nonelementary \ca{} with locally finite nuclear dimension.
Then strict comparison, pureness, $\mathcal{Z}$-stability, and finite nuclear dimension are equivalent.
In this case, $\dim_{\mathrm{nuc}}(A)\leq1$.

In particular, the Toms--Winter conjecture holds for separable, simple, nonelementary, approximately subhomogeneous \ca{s}.
\end{thmIntro}

%==========================================================================================
The general results are based on the Global Glimm Property, as studied in
\cite{ThiVil23Glimm}.
We first consider the rank problem abstractly for \CuSgp{s}.
Given a \CuSgp{} $S$ and a compact subset $K\subseteq\F(S)$, we show that
density of ranks on $K$ is equivalent to almost divisibility of ranks on $K$;
see \autoref{prp:CharDenseK}.
We also give a sufficient criterion for these conditions in terms of a weak form
of $(2,\omega)$-divisibility together with a smoothing property.

We then specialize these results to Cuntz semigroups of \ca{s}.
For every \ca{} $A$, we show that $\Cu(A)$ satisfies the required smoothing
property with respect to the cone $\T(A)$ of lower semicontinuous traces.
As a consequence, density of ranks on traces, almost divisibility of ranks on
traces, and the tracial Global Glimm Property are equivalent;
see \autoref{prp:CharDenseTA}.
In particular, the Global Glimm Property implies that ranks are dense and almost
divisible on traces.

%==========================================================================================
\begin{thmIntro}[\autoref{prp:GGP}]
Let $A$ be a \ca{} with the Global Glimm Property.
Then ranks are dense and almost divisible on traces.
\end{thmIntro}

%==========================================================================================
Combining this with strict comparison by traces upgrades almost divisibility of
ranks to almost divisibility in the Cuntz semigroup and yields the general
pureness theorem.

%==========================================================================================
\begin{thmIntro}[\autoref{prp:GGPStrComp-Pure}, \autoref{prp:CharExactPure}]
Every \ca{} with the Global Glimm Property and strict comparison by traces is pure.
Consequently, an exact \ca{} is pure if and only if it has strict comparison and
the Global Glimm Property.
\end{thmIntro}

%==========================================================================================
\subsection*{Acknowledgments} 

%==========================================================================================
At an initial stage of this project, I used interactive collaboration with AI to investigate the rank problem for simple ASH algebras. 
This led to the key idea contained in \autoref{prp:CompressUnital} and \autoref{prp:Rounding}. 
The extension to the setting of the Global Glimm Property was developed by the author. 
I take full responsibility for the correctness of the arguments and the exposition.

%==========================================================================================
%==========================================================================================
\section{Ranks in simple unital \ca{s}}
\label{sec:simple-unital}

%==========================================================================================
Let $A$ be a unital, simple, nonelementary \ca{}.
We show that ranks of operators in the stabilization $A\otimes\Cpct$ are uniformly dense in $\Aff(T_1(A))_{++}$;
see \autoref{prp:RanksSimpleUnital}.
We deduce that strict comparison by traces implies almost divisibility, and hence pureness;
see \autoref{prp:PureSimpleUnital}.
As an application, we verify the Toms--Winter conjecture for \ca{s} with locally finite nuclear dimension, and in particular for approximately subhomogeneous \ca{s};
see \autoref{prp:TomsWinter} and \autoref{rmk:ASH}.

\medskip

For a \ca{} $A$, we write $A_+$ for its positive cone and $T_1(A)$ for its tracial state space, equipped with the weak* topology.
We extend traces to matrix algebras using the unnormalized matrix trace, and use the same symbol for these extensions.
Given a lower semicontinuous trace $\tau\colon A_+\to[0,\infty]$ and $a\in M_m(A)_+$, the \emph{rank} of $a$ at $\tau$ is
\[
d_\tau(a):=\lim_{n\to\infty}\tau(a^{1/n}).
\]
We write $\hat{a}(\tau):=d_\tau(a)$.
For $t>0$, the cutdown $(a-t)_+$ is defined by continuous functional calculus.

Let $\mathcal{K}$ denote the compact operators on a separable infinite-dimensional Hilbert space.
For $a,b\in(A\otimes\Cpct)_+$, we say that $a$ is \emph{Cuntz subequivalent} to $b$, and write $a\precsim b$, if there are $r_n\in A\otimes\mathcal{K}$ such that $r_n^*br_n\to a$ in norm.
We say that $a$ is \emph{Cuntz equivalent} to $b$, and write $a\sim b$, if $a \precsim b$ and $b \precsim a$.

We identify finite matrix algebras over $A$ with corners in $A\otimes\mathcal{K}$.
Lower semicontinuous traces extend canonically to $A\otimes\mathcal{K}$, again without normalizing the matrix traces.
We will use that $a\precsim b$ implies $d_\tau(a)\leq d_\tau(b)$.
For $b\in A_+$, set $\Her(b):=\overline{bAb}$.

%==========================================================================================
\begin{lma}
\label{prp:CompressUnital}
Let $A$ be a unital \ca{}, let $b \in A_+$ be full, and let $a \in A_+$.
Then there exists $c \in\Her(b)_+$ such that $\tau(c)=\tau(a)$ for all $\tau \in T_1(A)$.
\end{lma}
\begin{proof}
Since $b$ is full, there exist $y_1,\ldots,y_N\in A$ such that
\[
\sum_{j=1}^N y_j^*by_j=1.
\]
Set
\[
c 
:= \sum_{j=1}^N b^{\frac{1}{2}}y_j a y_j^*b^{\frac{1}{2}} \in \Her(b)_+.
\]
Then
\[
\tau(c)
= \sum_{j=1}^N \tau\bigl( b^{\frac{1}{2}}y_j a y_j^*b^{\frac{1}{2}} \bigr)
= \sum_{j=1}^N \tau\bigl( a^{\frac{1}{2}}y_j^*by_j a^{\frac{1}{2}} \bigr)
= \tau(a)
\]
for every $\tau \in T_1(A)$.
\end{proof}

%==========================================================================================
The next estimate converts an evaluation into a rank function.
Its error is bounded by the rank of the original element, independently of its norm.
We state it for traces that may take infinite values, as this will be needed later.

%==========================================================================================
\begin{lma}
\label{prp:Rounding}
Let $A$ be a \ca{}, let $c \in A_+$, and let $m \geq\max\{1,\|c\|\}$ be an integer.
Set
\[
e := \bigoplus_{j=1}^m (c-j)_+\in M_m(A)_+.
\]
Then
\[
d_\tau(e)
\leq \tau(c)
\leq d_\tau(e)+d_\tau(c)
\]
for every lower semicontinuous trace $\tau\colon A_+\to[0,\infty]$.
\end{lma}
\begin{proof}
If $\tau(c)=\infty$, then $d_\tau(c)=\infty$, since
$\tau(c)\leq m d_\tau(c)$, and the inequalities hold.
Assume that $\tau(c)<\infty$.

Let $\spec(c)$ denote the spectrum of~$c$.
Functional calculus identifies $C^*(c)$ with
$C_0(\spec(c)\setminus\{0\})$.
The restriction of~$\tau$ to $C^*(c)$ is a lower semicontinuous weight,
and therefore is given by integration against a regular Borel measure~$\mu$
on $\spec(c)\setminus\{0\}$;
see \cite[Example~II.6.7.2(v)]{Bla06OpAlgs}.
We extend~$\mu$ to a measure on~$\RR$ by declaring the complement of
$\spec(c)\setminus\{0\}$ to have measure zero.

Then functional calculus gives
\[
\tau(c)
= \int s\,d\mu(s),
\qquad
d_\tau(c)
= \int \charFct_{(0,\infty)}(s) \,d\mu(s),
\]
and
\[
d_\tau((c-j)_+)
= \int \charFct_{(j,\infty)}(s) \,d\mu(s).
\]
Since $\spec(c)\subseteq[0,m]$, for $s\in[0,m]$ one has
\[
\sum_{j=1}^m \mathbf{1}_{(j,\infty)}(s)
\leq s
\leq\sum_{j=1}^m \mathbf{1}_{(j,\infty)}(s)
+\mathbf{1}_{(0,\infty)}(s).
\]
Integrating these inequalities and using additivity of rank on direct sums proves the result.
\end{proof}

%==========================================================================================
For a compact convex set~$K$, we let $\Aff(K)$ denote the space of continuous affine functions from~$K$ to~$\RR$, and write $\Aff(K)_{++}$ for the set of strictly positive functions in~$\Aff(K)$.

A \ca{} is said to be \emph{elementary} if it is isomorphic to the compact operators on a Hilbert space.

%==========================================================================================
\begin{thm}
\label{prp:RanksSimpleUnital}
Let $A$ be a unital, simple, nonelementary \ca{}.
Assume that $T_1(A)\neq\emptyset$, and let $f\in\Aff(T_1(A))_{++}$ and $\varepsilon>0$.
Then there exist $m\geq1$ and $e\in M_m(A)_+$ such that
\[
\bigl| f(\tau)-d_\tau(e) \bigr| \leq \varepsilon
\]
for every $\tau \in T_1(A)$.
\end{thm}
\begin{proof}
By \cite[Corollary~3.10]{BroPerTom08CuElliottConj}, which is based on \cite{CunPed79EquivTraces}, there exists $a \in A_+$ such that $f(\tau)=\tau(a)$ for every $\tau \in T_1(A)$.

Choose an integer $n\geq2$ such that $\tfrac{1}{n} \leq \varepsilon$.
Since $A$ is simple and nonelementary, it has an infinite-dimensional irreducible representation.
Glimm's lemma, in the form of \cite[Proposition~3.10]{RobRor13Divisibility}, gives a nonzero $*$-homomorphism
\[
\varphi\colon C_0((0,1])\otimes M_n\longrightarrow A.
\]
Let $\iota\in C_0((0,1])$ be the function $\iota(t)=t$, and set $b_j:=\varphi(\iota\otimes e_{jj})$ for $j=1,\ldots,n$.
These elements are nonzero, pairwise orthogonal, and Cuntz equivalent.
Thus, for $b:=b_1$, we have
\[
nd_\tau(b)=d_\tau\Bigl(\sum_{i=1}^n b_i\Bigr)\leq1
\]
for every $\tau \in T_1(A)$.
The element $b$ is full by simplicity of $A$.

Apply \autoref{prp:CompressUnital} to obtain $c\in\Her(b)_+$ with $\tau(c)=\tau(a)$ for all $\tau \in T_1(A)$.
Since  $c\in\Her(b)_+$, we have $c \precsim b$ and therefore
\[
d_\tau(c) \leq d_\tau(b) \leq \frac{1}{n} \leq \varepsilon
\]
for all $\tau \in T_1(A)$.

Choose an integer $m \geq \max\{1,\|c\|\}$ and apply \autoref{prp:Rounding}.
The resulting element $e \in M_m(A)_+$ satisfies
\[
d_\tau(e)
\leq \tau(c) %= \tau(a) = f(\tau)
\leq d_\tau(e) + d_\tau(c)
\leq d_\tau(e) + \varepsilon
\]
for every $\tau \in T_1(A))$.
Since $\tau(c)=\tau(a)=f(\tau)$, the conclusion follows.
\end{proof}

%==========================================================================================
\begin{rmk}
The proof of \autoref{prp:RanksSimpleUnital} only uses simplicity and nonelementarity to obtain full elements of arbitrarily small rank uniformly on $T_1(A)$.
Thus the conclusion holds for every unital \ca{} with this property.
\end{rmk}

%==========================================================================================
We use the Cuntz semigroup $\Cu(A)$ of a \ca{} $A$ with its usual order, addition, and sequential suprema;
see \cite{CowEllIva08CuInv,AntPerThi18TensorProdCu}.
Its elements are Cuntz equivalence classes $[a]$ of positive elements in $A\otimes\Cpct$.
It is known that increasing sequences in $\Cu(A)$ have suprema.
For $x',x\in\Cu(A)$, we write $x'\ll x$ if, whenever $x\leq\sup_n x_n$ for an increasing sequence $(x_n)_n$, one has $x'\leq x_n$ for some $n$.
Then $[(a-\varepsilon)_+]\ll[a]$ for every $a \in (A\otimes\Cpct)_+$ and $\varepsilon>0$.
Further, we have $[a]=\sup_n[(a-\tfrac{1}{n})_+]$, and in fact every element in $\Cu(A)$ is the supremum of a $\ll$-increasing sequence.

For a lower semicontinuous trace $\tau$ on~$A$ and $x \in \Cu(A)$, we set $\widehat{x}(\tau):=d_\tau(a)$, which is independent of the representative $a \in (A\otimes\Cpct)_+$ when $x=[a]$.

The semigroup $\Cu(A)$ is \emph{almost unperforated} if for all $x,y \in \Cu(A)$ and  $k\geq 1$, if $(k+1)x\leq ky$ then $x\leq y$.
It is \emph{almost divisible} if for $x'\ll x$ in $\Cu(A)$ and $k\geq1$, there exists $y\in\Cu(A)$ such that
\[
ky\leq x
\andSep
x'\leq(k+1)y.
\]
The \ca{} $A$ is \emph{pure} if $\Cu(A)$ is almost unperforated and almost divisible.

We use $\QT(A)$ to denote the cone of lower semicontinuous $2$-quasitraces on~$A$.
It is known that $\Cu(A)$ is almost unperforated if and only if $A$ has
\emph{strict comparison}, in the sense that $a,b\in(A\otimes\Cpct)_+$
satisfy $a\precsim b$ whenever there exists $\gamma\in(0,1)$ such that $d_\tau(a)\leq \gamma d_\tau(b)$ for every $\tau\in\QT(A)$;
see \cite[Proposition~4.2]{EllRobSan11Cone}, and also
\cite[Paragraph~3.4]{AntPerThiVil24arX:PureCAlgs}.

Following Ng--Robert \cite[Definition~3.1]{NgRob16CommutatorsPureCa}, we say that $A$ has \emph{strict comparison by traces} if $a,b\in(A\otimes\Cpct)_+$
satisfy $a\precsim b$ whenever there exists $\gamma\in(0,1)$ such that $d_\tau(a)\leq \gamma d_\tau(b)$ for every lower semicontinuous trace~$\tau$ on~$A$.
By \cite[Theorem~3.6, Remark~3.7]{NgRob16CommutatorsPureCa}, a \ca{} $A$ has strict comparison by traces if and only if $\QT(A)=\T(A)$ and $A$ has strict comparison.

For exact \ca{s}, every lower semicontinuous $2$-quasitrace is a trace;
see \cite{Haa14Quasitraces} and the discussion of the nonunital case in \cite[Section~2]{TikTom15CuSgpNonunital}.

%==========================================================================================
The next result can be deduced from \cite[Proposition~2.1]{Tom11CharClassifiableAH}, or from \cite[Proposition~2.11]{Thi20RksOps}.
We include a proof for the convenience of the reader.

%==========================================================================================
\begin{cor}
\label{prp:PureSimpleUnital}
Let $A$ be a unital, simple, nonelementary \ca{} with strict comparison by traces.
Then $A$ is pure.
\end{cor}
\begin{proof}
In general, strict comparison by traces implies strict comparison, and hence almost unperforation of the Cuntz semigroup.

Suppose first that $T_1(A)=\emptyset$.
Then $\T(A)=\{\tau_0,\tau_\infty\}$, where $\tau_0(a)=0$ for every
$a\in A_+$, while $\tau_\infty(0)=0$ and $\tau_\infty(a)=\infty$ for every
nonzero $a\in A_+$.
Thus, for any nonzero $a,b\in(A\otimes\Cpct)_+$ and every $\tau\in\T(A)$, we have
\[
d_\tau(a)\leq \tfrac{1}{2}d_\tau(b)=d_\tau(b).
\]
Strict comparison by traces therefore gives $a\precsim b$.
By symmetry, $a\sim b$, so all nonzero positive elements are Cuntz equivalent.
Hence $\Cu(A)\cong\{0,\infty\}$, and $A$ is purely infinite, and in particular pure.

Now assume that $T(A)\neq\emptyset$.
Fix $x'\ll x$ in $\Cu(A)$ and $k\geq1$.
If $x'=0$, take $y=0$.
Using that Cuntz classes of elements in matrix algebras over $A$ are sup-dense in $\Cu(A)$, choose $a \in M_l(A)_+$ such that
\[
x'\ll [a] \leq x.
\]
Then choose $\varepsilon>0$ such that
\[
x' \leq [(a-\varepsilon)_+].
\]

Choose a continuous function $g\colon[0,\infty)\to[0,1]$ that vanishes at~$0$
and equals one on $[\varepsilon,\infty)$.
Define $f_0\colon T_1(A)\to\RR$ by $f_0(\tau):=\tau(g(a))$.
Since $g(a)\neq0$ and every tracial state on the simple \ca{} $A$ is faithful,
we have $f_0\in\Aff(T_1(A))_{++}$.
Moreover,
\[
d_\tau((a-\varepsilon)_+)
\leq f_0(\tau)
\leq d_\tau(a)
\]
for every $\tau\in T_1(A)$.
Set
\[
m:=\min_{\tau\in T_1(A)}f_0(\tau)>0,
\qquad
s:=k+\tfrac12.
\]

Apply \autoref{prp:RanksSimpleUnital} for the strictly positive, continuous, affine functions $f:=\frac{1}{s}f_0$ and tolerance $\delta:=\frac{m}{4s(k+1)}$ to obtain $e\in(A\otimes\Cpct)_+$ such that
\[
\bigl|f(\tau)-d_\tau(e)\bigr|\leq\delta
\]
for every $\tau\in T_1(A)$.
Set $y:=[e]$.

Since $m\leq f_0(\tau)$, we have
\[
d_\tau(e)
\leq
\frac{1}{s}\left(1+\frac{1}{4(k+1)}\right)f_0(\tau)
\leq
\frac{1}{s}\left(1+\frac{1}{4(k+1)}\right)d_\tau(a).
\]
Thus
\[
k d_\tau(e)
\leq
\gamma_1 d_\tau(a),
\quad\text{for }
\gamma_1:=
\frac{k}{s}\left(1+\frac{1}{4(k+1)}\right)<1.
\]
Similarly,
\[
d_\tau(e)
\geq
\frac{1}{s}\left(1-\frac{1}{4(k+1)}\right)f_0(\tau)
=
\frac{k+\frac34}{s(k+1)}f_0(\tau),
\]
and hence
\[
d_\tau((a-\varepsilon)_+)
\leq
\gamma_2(k+1)d_\tau(e),
\quad\text{for }
\gamma_2:=
\frac{s}{k+\frac34}<1.
\]
These inequalities extend by homogeneity to all lower semicontinuous traces
on~$A$.
Strict comparison by traces therefore gives
\[
ky\leq[a]\leq x
\andSep
x'\leq[(a-\varepsilon)_+]\leq(k+1)y.
\]
Thus $\Cu(A)$ is almost divisible.
\end{proof}

%==========================================================================================
A \ca{} has \emph{locally finite nuclear dimension} if each finite subset can be
approximated arbitrarily well by a sub-\ca{} of finite nuclear dimension.
Every \ca{} with locally finite nuclear dimension is nuclear.

We now combine \autoref{prp:PureSimpleUnital} with the known absorption and
dimension theorems to verify the Toms--Winter conjecture for \ca{s} with locally
finite nuclear dimension, and in particular for approximately subhomogeneous
(ASH) \ca{s}.

%==========================================================================================
\begin{thm}
\label{prp:TomsWinter}
Let $A$ be a separable, unital, simple, nonelementary \ca{} with locally finite
nuclear dimension.
Then the following are equivalent:
\begin{enumerate}[label=\textup{(\arabic*)}]
\item $A$ has finite nuclear dimension;
\item $A$ is $\calZ$-stable;
\item $A$ has strict comparison;
\item[$(3')$] $A$ is pure.
\end{enumerate}
In this case, $\dim_{\mathrm{nuc}}(A)\leq1$.
\end{thm}

\begin{proof}
The implications (1)$\Rightarrow$(2) and (3')$\Rightarrow$(2) were shown in
\cite[Corollary~7.3, Theorem~7.1]{Win12NuclDimZstable}.
By \cite{CasEviTikWhiWin21NucDimSimple}, $\calZ$-stability implies nuclear
dimension at most one, and hence (2)$\Rightarrow$(1).
The implication (2)$\Rightarrow$(3') was shown in \cite{Ror04StableRealRankZ}; see also \cite[Proposition~5.2]{AntPerThiVil24arX:PureCAlgs}.
Clearly, (3') implies (3).

It remains to show that (3) implies (3').
Since $A$ has locally finite nuclear dimension, it is nuclear and hence exact.
Thus every lower semicontinuous $2$-quasitrace on~$A$ is a trace, so strict
comparison agrees with strict comparison by traces.
The implication (3)$\Rightarrow$(3') therefore follows from
\autoref{prp:PureSimpleUnital}.
\end{proof}

%==========================================================================================
\begin{rmk}
\label{rmk:ASH}
Approximately subhomogeneous \ca{s} have locally finite decomposition rank by
\cite[Corollary~2.2]{NgWin06NoteSH}, and hence locally finite nuclear dimension.
Thus \autoref{prp:TomsWinter} applies, in particular, to separable, unital, simple, nonelementary ASH algebras.
\end{rmk}

%==========================================================================================
%==========================================================================================
\section{Density and divisibility of ranks in Cu-semigroups}
\label{sec:RanksCu}

%==========================================================================================
We refer to \cite{AntPerThi18TensorProdCu} and \cite{GarPer26ModernCu} for an introduction to the theory of \CuSgp{s}.
For a \CuSgp{} $S$, let $\F(S)$ denote its cone of functionals, that is, the additive, order-preserving maps
$\lambda\colon S\to[0,\infty]$ that preserve zero and suprema of increasing sequences.
We equip $\F(S)$ with its standard compact Hausdorff cone topology.
For $x\in S$, we write
\[
\widehat{x}(\lambda):=\lambda(x),
\qquad \lambda\in\F(S).
\]

Let $\Lsc(\F(S))$ denote the cone of lower semicontinuous, additive, homogeneous,
order-preserving functions $\F(S)\to[0,\infty]$ that vanish at zero.
If $S$ satisfies \axiomO{5}, we denote by $\LL(\F(S))$ its realification,
equivalently the smallest subsemigroup of $\Lsc(\F(S))$ that is closed under
suprema of increasing sequences and contains all functions of the form
$\frac{1}{n}\widehat{x}$, for $x\in S$ and $n\geq1$;
see \cite[Section~3]{Rob13Cone}.
The semigroup $\LL(\F(S))$ is again a \CuSgp{} satisfying \axiomO{5}.

For $S=\Cu(A)$, functionals correspond naturally to lower semicontinuous
$2$-quasitraces on $A$;
see \cite[Proposition~4.2]{EllRobSan11Cone}.

%==========================================================================================
\begin{prp}
\label{prp:CharDenseK}
Let $S$ be a \CuSgp{} satisfying \axiomO{5}, and let $K\subseteq\F(S)$ be compact.
Then the following are equivalent:
\begin{enumerate}
\item
Ranks are dense on $K$:
for every $f' \ll f$ in $\LL(\F(S))$, there exists $y \in S$ such that
\[
f'|_K \leq \widehat{y}|_K \leq f|_K.
\]
\item
For every $x' \ll x$ in $S$ and every $k \geq 2$, there exist
$\varepsilon\in(0,1)$ and $y \in S$ such that
\[
k\widehat{y}|_K \leq (1-\varepsilon)\widehat{x}|_K,
\andSep
\widehat{x'}|_K \leq (1-\varepsilon)(k+1)\widehat{y}|_K.
\]
\item
Ranks are almost divisible on $K$:
for every $x' \ll x$ in $S$ and every $k \geq 2$, there exists $y \in S$ such that
\[
k\widehat{y}|_K \leq \widehat{x}|_K,
\andSep
\widehat{x'}|_K \leq (k+1)\widehat{y}|_K.
\]
\end{enumerate}
\end{prp}
%------------------------------------------------------------------------------------------
\begin{proof}
We first show that~(1) implies~(2).
Let $x' \ll x$ in $S$, and let $k \geq 2$.
Choose $\varepsilon\in(0,1)$ such that $\frac{k}{k+1}<(1-\varepsilon)^2$.
Then
\[
\frac{1}{(1-\varepsilon)(k+1)}
<
\frac{1-\varepsilon}{k}.
\]
Since $x'\ll x$, it follows from \cite[Lemma~2.2.5]{Rob13Cone} that
\[
\frac{1}{(1-\varepsilon)(k+1)}\widehat{x'}
\ll
\frac{1-\varepsilon}{k}\widehat{x}
\]
in $\Lsc(\F(S))$, and hence in $\LL(\F(S))$.
By~(1), there exists $y\in S$ such that
\[
\frac{1}{(1-\varepsilon)(k+1)}\widehat{x'}|_K
\leq \widehat{y}|_K
\leq
\frac{1-\varepsilon}{k}\widehat{x}|_K.
\]
Thus
\[
k\widehat{y}|_K
\leq (1-\varepsilon)\widehat{x}|_K,
\andSep
\widehat{x'}|_K
\leq (1-\varepsilon)(k+1)\widehat{y}|_K,
\]
as desired.

It is clear that~(2) implies~(3).

\medskip

We show that~(3) implies~(1).
Let $f' \ll f$ in $\LL(\F(S))$.
Since $\LL(\F(S))$ is a \CuSgp, we can choose $f'' \in \LL(\F(S))$ such that $f' \ll f'' \ll f$.
Using that functions of the form $\tfrac{1}{k}\widehat{x}$, for $x \in S$ and $k \geq 2$, are sup-dense in $\LL(\F(S))$, we obtain $x \in S$ and $k \geq 2$
such that
\[
f'' \ll \frac{1}{k}\widehat{x} \ll f.
\]
Since
\[
f'
\ll \frac{1}{k}\widehat{x}
= \sup_m \frac{m}{m+1}\frac{1}{k}\widehat{x},
\]
there exists $m \geq 1$ such that
\[
f''
\leq \frac{m}{m+1}\frac{1}{k}\widehat{x}
\leq \frac{m}{mk+1}\widehat{x}.
\]

Choose a $\ll$-increasing sequence $(x_n)_n$ in $S$ with supremum $x$.
Since
\[
f'
\ll f''
\leq \frac{m}{mk+1}\widehat{x}
=
\sup_n \frac{m}{mk+1}\widehat{x_n},
\]
there exists $n$ such that
\[
f' \leq \frac{m}{mk+1}\widehat{x_n}.
\]

Now $mx_n \ll mx$.
Applying~(3) to $mx_n \ll mx$ and $mk$, we obtain $y \in S$ such that
\[
mk\widehat{y}|_K \leq m\widehat{x}|_K,
\andSep
m\widehat{x_n}|_K \leq (mk+1)\widehat{y}|_K.
\]
Consequently,
\[
f'|_K
\leq \frac{m}{mk+1}\widehat{x_n}|_K
\leq \widehat{y}|_K
\leq \frac{1}{k}\widehat{x}|_K
\leq f|_K,
\]
as desired.
\end{proof}

%==========================================================================================
The first condition in the next result is a weakening of $(2,\omega)$-divisibility, where the comparison $2t\leq x$ is replaced by $2\widehat{t}|_K \leq \widehat{x}|_K$.
The second condition is a smoothing property.

%==========================================================================================
\begin{lma}
\label{prp:SufficientDenseK}
Let $S$ be a \CuSgp{} satisfying \axiomO{5}, and let $K\subseteq\F(S)$ be compact.
Assume that $S$ satisfies the following conditions:
\begin{enumerate}
\item
For every $x' \ll x$ in $S$, there exists $t \in S$ such that
\[
2\widehat{t}|_K
\leq \widehat{x}|_K,
\andSep
x' \leq \infty t.
\]

\item
For every $x' \ll x \leq \infty t$ in $S$ and every $k \geq 2$,
there exists $y \in S$ such that
\[
k\widehat{y}|_K
\leq \widehat{x}|_K,
\andSep
\widehat{x'}|_K
\leq k\widehat{y}|_K + k\widehat{t}|_K.
\]
\end{enumerate}
Then $S$ satisfies the equivalent statements of \autoref{prp:CharDenseK}.
\end{lma}
%------------------------------------------------------------------------------------------
\begin{proof}
We first observe that condition~(1) can be iterated.
We claim that, for every $x' \ll x$ in $S$ and every $n\geq1$, there exists
$t\in S$ such that
\[
2^n\widehat{t}|_K\leq\widehat{x}|_K,
\andSep
x'\leq\infty t.
\]
We prove this by induction on $n$.
The case $n=1$ is condition~(1).

Assume that the assertion holds for some $n\geq1$.
To verify if for $n+1$, let $x' \ll x$ in $S$.
Choose $x''\in S$ such that
\[
x'\ll x''\ll x.
\]
By the induction hypothesis, applied for $x'' \ll x$, there exists $s\in S$ such that
\[
2^n\widehat{s}|_K\leq\widehat{x}|_K,
\andSep
x''\leq\infty s.
\]
Using that $x'\ll\infty s$, choose $s' \in S$ such that
\[
s' \ll s, \andSep
x'\leq \infty s'.
\]
Applying condition~(1) to $s' \ll s$, we obtain $t\in S$ such that
\[
2\widehat{t}|_K\leq\widehat{s}|_K,
\andSep
s'\leq\infty t.
\]
It follows that
\[
2^{n+1}\widehat{t}|_K
\leq2^n\widehat{s}|_K
\leq\widehat{x}|_K, \andSep
x'\leq \infty s'\leq\infty t.
\]
This proves the claim.

\medskip

We now verify condition~(3) of \autoref{prp:CharDenseK}.
Let $x'\ll x$ in $S$, and let $k\geq2$.
Choose $x''\in S$ and $n\geq1$ such that
\[
x'\ll x''\ll x, \andSep
2^n\geq k(k+1).
\]
By the claim, there exists $t\in S$ such that
\[
k(k+1)\widehat{t}|_K
\leq\widehat{x}|_K,
\andSep
x''\leq\infty t.
\]
Applying condition~(2) to $x'\ll x''\leq\infty t$ and $k+1$, we obtain
$y\in S$ such that
\[
(k+1)\widehat{y}|_K
\leq\widehat{x''}|_K
\leq\widehat{x}|_K, \andSep
\widehat{x'}|_K
\leq
(k+1)\widehat{y}|_K+(k+1)\widehat{t}|_K.
\]
Set $z:=y+t$.
Then
\[
k\widehat{z}|_K
=
k\widehat{y}|_K+k\widehat{t}|_K
\leq
\frac{k}{k+1}\widehat{x}|_K
+
\frac{1}{k+1}\widehat{x}|_K
=
\widehat{x}|_K,
\]
while
\[
\widehat{x'}|_K
\leq
(k+1)\widehat{y}|_K+(k+1)\widehat{t}|_K
=
(k+1)\widehat{z}|_K.
\]
as desired.
\end{proof}

%==========================================================================================
%==========================================================================================
\section{Density and divisibility of ranks on the trace space}
\label{sec:RanksCAlg}

%==========================================================================================
We show that Cuntz semigroups of \ca{s} automatically satisfy the smoothing condition from \autoref{prp:SufficientDenseK}.
We deduce a characterization of density of tracial ranks.

%==========================================================================================
\begin{prp}
\label{prp:SmoothingTA}
Let $A$ be a \ca{}.
Then $\Cu(A)$ satisfies the following smoothing property with respect to the cone
$\T(A)$ of lower semicontinuous traces:
For every $x' \ll x \leq \infty t$ in $\Cu(A)$ and every $k \geq 2$,
there exists $y \in \Cu(A)$ such that
\[
k\widehat{y}|_{\T(A)}
\leq \widehat{x}|_{\T(A)},
\andSep
\widehat{x'}|_{\T(A)}
\leq k\widehat{y}|_{\T(A)} + k\widehat{t}|_{\T(A)}.
\]
\end{prp}
%------------------------------------------------------------------------------------------
\begin{proof}
Replacing $A$ by $A\otimes\Cpct$, we may assume that $A$ is stable.
Choose $a,b\in A_+$ such that
\[
[a]=x
\andSep
[b]=t,
\]
and choose $\varepsilon>0$ such that
\[
x'\leq[(a-\varepsilon)_+].
\]
Let $g\colon[0,\infty)\to[0,1]$ be a continuous function that vanishes on a
neighborhood of zero and takes the value $1$ on $[\varepsilon,\infty)$.
Then
\[
\widehat{x'}(\tau)
\leq d_\tau((a-\varepsilon)_+)
\leq \tau(g(a))
\leq d_\tau(a)
=\widehat{x}(\tau)
\]
for every $\tau\in\T(A)$.

Let $J_b$ denote the closed ideal generated by $b$.
Since $x\leq\infty t$, we have $a\in J_b$, and therefore $g(a)\in\Ped(J_b)$.
By \cite[Corollary~II.5.2.11]{Bla06OpAlgs}, there exist $r_1,\ldots,r_N\in A$ such that
\[
g(a)=\sum_{j=1}^N r_j^*br_j.
\]
Set
\[
c := \sum_{j=1}^N b^{\frac{1}{2}}r_jr_j^*b^{\frac{1}{2}}.
\]
Then $c \in \overline{bAb}$, and hence $c \precsim b$.
Moreover, additivity and the trace identity give
\[
\tau(c)
= \sum_{j=1}^N \tau\bigl( b^{\frac{1}{2}}r_jr_j^*b^{\frac{1}{2}} \bigr)
= \sum_{j=1}^N \tau\bigl( r_j^*br_j \bigr)
=\tau(g(a))
\]
for every $\tau\in\T(A)$, also when these values are infinite.

Apply \autoref{prp:Rounding} to $\tfrac{1}{k}c$, and let $y\in\Cu(A)$ be the class
of the resulting element.
For every $\tau\in\T(A)$, we obtain
\[
\widehat{y}(\tau)
\leq \frac{1}{k}\tau(c)
\leq \widehat{y}(\tau)+d_\tau(\tfrac{1}{k}c).
\]
Since $\tfrac{1}{k}c \sim c \precsim b$, we have
\[
d_\tau(\tfrac{1}{k}c) \leq d_\tau(b)=\widehat{t}(\tau).
\]
It follows that
\[
k\widehat{y}(\tau)
\leq \tau(c)
= \tau(g(a))
\leq \widehat{x}(\tau),
\]
and
\[
\widehat{x'}(\tau)
\leq \tau(g(a))
=\tau(c)
\leq k\widehat{y}(\tau)+k\widehat{t}(\tau).
\]
This proves the result.
\end{proof}

%==========================================================================================
\begin{thm}
\label{prp:CharDenseTA}
Let $A$ be a \ca{}.
Then the following are equivalent:
\begin{enumerate}
\item
Ranks are dense on traces:
For every $f' \ll f$ in $\LL(\QT(A))$, there exists
$a \in (A\otimes\Cpct)_+$ such that
\[
f'(\tau) \leq d_\tau(a) \leq f(\tau)
\]
for every $\tau \in \T(A)$.

\item
Ranks are almost divisible on traces:
For every $a \in (A\otimes\Cpct)_+$, every $\varepsilon>0$, and every
$k \geq 2$, there exists $b \in (A\otimes\Cpct)_+$ such that
\[
kd_\tau(b) \leq d_\tau(a),
\andSep
d_\tau((a-\varepsilon)_+) \leq (k+1)d_\tau(b)
\]
for every $\tau \in \T(A)$.

\item
$A$ has the tracial Global Glimm Property:
For every $a \in (A\otimes\Cpct)_+$ and every $\varepsilon>0$,
there exists $b \in (A\otimes\Cpct)_+$ such that
\[
2d_\tau(b) \leq d_\tau(a) \text{ for every $\tau \in \T(A)$}, \andSep
[(a-\varepsilon)_+] \leq \infty[b].
\]
\end{enumerate}
\end{thm}
%------------------------------------------------------------------------------------------
\begin{proof}
The equivalence between~(1) and~(2) follows from \autoref{prp:CharDenseK}, applied to $S=\Cu(A)$ and $K=\T(A)\subseteq\QT(A)=\F(\Cu(A))$.

We show that~(2) implies~(3).
Let $a\in(A\otimes\Cpct)_+$ and $\varepsilon>0$.
Applying~(2) with $k=2$, we obtain $b\in(A\otimes\Cpct)_+$ such that
\[
2d_\tau(b)\leq d_\tau(a),
\andSep
d_\tau((a-\varepsilon)_+)\leq3d_\tau(b)
\]
for every $\tau\in\T(A)$.
Since lower-semicontinuous traces detect closed ideals, it follows that $[(a-\varepsilon)_+]\leq\infty[b]$, as desired.

Finally, assume~(3).
We verify the hypotheses of \autoref{prp:SufficientDenseK} for $S=\Cu(A)$ and $K=\T(A)$.
The first condition of \autoref{prp:SufficientDenseK} holds by~(3).
The second condition is \autoref{prp:SmoothingTA}.
Therefore, \autoref{prp:SufficientDenseK} shows that~(1) and~(2) hold.
\end{proof}

%==========================================================================================
A \ca{} $A$ has the \emph{Global Glimm Property} if, for every $a\in A_+$, $n\geq2$, and $\varepsilon>0$, there is a $*$-homomorphism
\[
\varphi\colon C_0((0,1])\otimes M_n\longrightarrow\Her(a)
\]
such that $(a-\varepsilon)_+$ belongs to the closed ideal generated by the image of $\varphi$.
By \cite[Theorem~3.6]{ThiVil23Glimm}, this is equivalent to $(2,\omega)$-divisibility of $\Cu(A)$, and also to $(n,\omega)$-divisibility for every $n\geq2$.
Here $(n,\omega)$-divisibility means that, for every $x'\ll x$, there is $y$ with
\[
ny\leq x
\andSep
x'\leq\infty y:=\sup_{j\geq1}jy.
\]
The Global Glimm Property passes to stabilization by \cite[Proposition~3.12]{ThiVil23Glimm}.

%==========================================================================================
\begin{cor}
\label{prp:GGP}
Let $A$ be a \ca{} with the Global Glimm Property.
Then ranks are dense and almost divisible on traces in the sense of
\autoref{prp:CharDenseTA}.
\end{cor}
%------------------------------------------------------------------------------------------
\begin{proof}
The Global Glimm Property implies its tracial version.
\end{proof}

%==========================================================================================
\begin{thm}
\label{prp:GGPStrComp-Pure}
Every \ca{} with the Global Glimm Property and strict comparison by traces is pure.
\end{thm}
%------------------------------------------------------------------------------------------
\begin{proof}
Strict comparison by traces implies strict comparison, and hence $\Cu(A)$ is
almost unperforated.

It remains to show almost divisibility.
Let $x'\ll x$ in $\Cu(A)$, and let $k\geq2$.
By the Global Glimm Property and \autoref{prp:GGP}, ranks are almost divisible on traces.
Applying \autoref{prp:CharDenseK} with $S=\Cu(A)$ and $K=\T(A)$, we obtain $\varepsilon>0$ and $y\in\Cu(A)$ such that
\[
k\widehat{y}|_{\T(A)}
\leq (1-\varepsilon)\widehat{x}|_{\T(A)}, \andSep
\widehat{x'}|_{\T(A)}
\leq
(1-\varepsilon)(k+1)\widehat{y}|_{\T(A)}.
\]
Strict comparison by traces therefore gives
\[
ky\leq x
\andSep
x'\leq(k+1)y,
\]
as desired.
Hence $A$ is pure.
\end{proof}

%==========================================================================================
\begin{cor}
\label{prp:CharExactPure}
Let $A$ be an exact \ca{}.
Then $A$ is pure if and only if it has strict comparison and the Global Glimm Property.
\end{cor}
\begin{proof}
The forward implication holds in general:
every pure \ca{} has strict comparison and the Global Glimm Property.

Conversely, assume that $A$ has strict comparison and the Global Glimm Property.
Since $A$ is exact, every lower semicontinuous $2$-quasitrace on $A$ is a trace.
Thus strict comparison agrees with strict comparison by traces.
Hence $A$ is pure by \autoref{prp:GGPStrComp-Pure}.
\end{proof}

%==========================================================================================
%==========================================================================================
%\bibliographystyle{../../aomalphaMyShort}
%\bibliography{../../References}

\providecommand{\etalchar}[1]{$^{#1}$}
\providecommand{\bysame}{\leavevmode\hbox to3em{\hrulefill}\thinspace}
\providecommand{\noopsort}[1]{}
\providecommand{\mr}[1]{\href{http://www.ams.org/mathscinet-getitem?mr=#1}{MR~#1}}
\providecommand{\zbl}[1]{\href{http://www.zentralblatt-math.org/zmath/en/search/?q=an:#1}{Zbl~#1}}
\providecommand{\jfm}[1]{\href{http://www.emis.de/cgi-bin/JFM-item?#1}{JFM~#1}}
\providecommand{\arxiv}[1]{\href{http://www.arxiv.org/abs/#1}{arXiv~#1}}
\providecommand{\doi}[1]{\url{http://dx.doi.org/#1}}
\providecommand{\MR}{\relax\ifhmode\unskip\space\fi MR }
\providecommand{\MRhref}[2]{%
  \href{http://www.ams.org/mathscinet-getitem?mr=#1}{#2}
}
\providecommand{\href}[2]{#2}

\end{document}